\documentclass{amsart}
\usepackage{pdfsync}
\usepackage{xspace}
\usepackage[all]{xy}
\usepackage{enumerate}
\usepackage{amssymb}
\usepackage{amsmath}
\usepackage{tikz-cd}

\makeatletter
\DeclareRobustCommand*\cal{\@fontswitch\relax\mathcal}
\makeatother

\newtheorem{thm}{Theorem}
\newtheorem{definition}{Definition}
\newtheorem{rem}[definition]{Remark}
\newtheorem{exam}[definition]{Example}
\newtheorem{prop}[definition]{Proposition}
\newtheorem{lem}[definition]{Lemma}
\newtheorem{cor}[definition]{Corollary}
\newtheorem{fact}[definition]{Fact}
\newtheorem{quest}{Question}

\newcommand{\br}{\overline}
\newcommand{\A}{\mathcal A}

\newcommand{\B}{\mathcal B}
\newcommand{\C}{\mathcal C}

\newcommand{\Mod}{\mathop{\rm Mod}}

\newcommand{\RMod}{R\textrm{-Mod}}
\newcommand{\Rmod}{R\textrm{-mod}}
\newcommand{\ZMod}{\mathbb{Z}\textrm{-Mod}}

\newcommand{\RFlat}{R\textrm{-Flat}}

\newcommand{\RProj}{R\textrm{-Proj}}

\newcommand{\RFree}{R\textrm{-Free}}

\newcommand{\Add}{\mathop{{\rm Add}}\nolimits}
\newcommand{\add}{\mathop{{\rm add}}\nolimits}

\newcommand{\Ind}{\mathop{{\rm Ind}}\nolimits}

\newcommand{\Th}{\mathop{{\rm Th}}\nolimits}

\newcommand{\n}{\mathop{{\rm n}}\nolimits}

\renewcommand{\phi}{\varphi}

\renewcommand{\prec}{\preceq}

\makeatletter
\@namedef{subjclassname@2020}{\textup{2020} Mathematics Subject Classification}
\makeatother

\begin{document}
\title{Free algebras, universal models and Bass modules}
\author{Anand Pillay and Philipp Rothmaler}
\subjclass[2020]{Primary 	08B20, 	16D40; Secondary 	03C05, 	03C35, 	03C60, 13L05.}
\date{\today \hspace{1em}}
\maketitle
\begin{abstract} We investigate the question of when free structures of infinite rank (in a variety) possess model-theoretic properties like categoricity in higher power, saturation, or universality. Concentrating on left $R$-modules we show, among other things, that the free module of infinite rank $R^{(\kappa)}$ purely embeds every $\kappa$-generated flat left $R$-module iff $R$ is left perfect. Using a Bass module corresponding to a descending chain of principal right ideals, we construct a model of the theory $T$ of $R^{(\kappa)}$ whose projectivity is equivalent to left perfectness, which allows to add a "stronger" equivalent condition: $R^{(\kappa)}$ purely (equivalently, elementarily) embeds every $\kappa$-generated flat left $R$-module which is a model of $T$.

In addition,  we extend the model-theoretic construction of this Bass module to arbitrary descending chains of pp formulas, resulting in a `Bass theory' of pure-projective modules. We put this new theory to use by reproving an old result of Daniel Simson about pure-semisimple rings and Mittag-Leffler modules.

\end{abstract}

\footnotetext{The first author was partially supported by NSF grants DMS-1665035,
DMS-1760212, and DMS-2054271.}

\section{Introduction}
A structure $U$ is called \texttt{universal} in a class $\cal K$ of structures of the same similarity type if every member of  $\cal K$ of cardinality at most that of $U$ is embeddable in $U$. 
Ever since seminal work of Roland Fra\"\i ss\'e and Bjarni J\'onsson in the 1950s and Morley and Vaught's work of the early 1960s universal structures have been part of both universal algebra and model theory. The latter work's greatest achievement is the development of a powerful machinery for the case where $\cal K$ is the class of models of a complete elementary theory $T$, equivalently, the class of structures elementarily equivalent to a single structure. Here one can strengthen the concept to elementary embeddings---the corresponding type of universal structure we call \texttt{elementarily universal} or  \texttt{$\prec$-universal} model (of $T$), see \cite{H} for references. The quest for universal members of classes of abelian groups and modules has spawned an active area of research starting with Paul Eklof's study of the early 1970s, often modified by an intermediate kind of embeddability, namely pure embeddability. The corresponding type of universal structures are \texttt{pure-universal} abelian groups or modules. For recent progress in this direction  see \cite{HM}, \cite{KM}, and \cite{M-A}.

All that work is concerned with the existence of universal structures in whatever classes under consideration. We here take a different approach. Instead of looking at the existence of universal structures we ask when specific structures are universal. These specific structures we take to be the free structures of infinite rank in a variety. %---in the sense of universal algebra. 
More concretely, let  $\cal V$ be a variety, i.e., an equationally axiomatized class of structures of given similarity type, and let $\cal K$ be a subclass of $\cal V$ containing all the free members of $\cal V$. We then ask when the free structures of infinite rank are universal for $\cal K$. This is directly inspired by corresponding work \cite{B-S}---and its follow-ups \cite{P-S} and \cite{KP}---about saturation of free structures.

The most interesting answers we have are for $\cal V=\RMod$, the variety of all, say left, modules over an arbitrary associative ring $R$,  and two particular classes $\cal K$, that of all flat modules in $\RMod$, denoted $\RFlat$, and the class $\cal M$ of all modules elementarily equivalent to a free module of infinite rank, which is the class of models of a complete theory $T$, for all the free left $R$-modules of infinite rank are elementarily equivalent. \\% $R^{(\omega)}$.\\

\noindent
\textbf{Theorem \ref{perfect}.} The following are equivalent for any ring $R$.
\vspace{-.1em}
\begin{enumerate}[\rm (i)]
\item $R$ is left perfect.
\item The free left $R$-module of some (any) infinite rank $\kappa$ is pure-universal among $\kappa$-generated flat left $R$-modules.
\item The free left $R$-module of some (any) infinite rank $\kappa$ is $\prec$--universal among $\kappa$-generated flat members of $\cal M$ (i.e., among $\kappa$-generated flat models of $T$).\\
\end{enumerate}

Note the third part, which says that as for left perfectness it suffices to examine the models of $T$. Bass' original description of left perfectness uses a test module, now called \texttt{Bass module}, denote it by $B$ accordingly, whose projectivity is equivalent to left perfectness. The novelty here is to use the same module $B$ to construct a model of\/ $T$, whose projectivity, again, entails left perfectness of the ring:\\

\noindent
\textbf{Lemma \ref{flatmodel}.} $R$ is left perfect if and only if every (countably generated)  flat model of $T$ is projective.\\

(Beware that the proofs of most statements about $T$, in particular that of of Thm.\  \ref{perfect} (in \S \ref{modules}),  therefore depend on the later \S \ref{Bass1}.)

Another novelty is a model-theoretic generalization of Bass' construction with respect to \emph{any} descending chain of pp formulas, Thm.\,\ref{pure} in \S \ref{purproj}. 
 Bass' original $B$ was based on a descending chain of principal right ideals.

Curiously, this generalization yields, as a special case, a  model-theoretic proof of an old algebraic result of Daniel Simson \cite{Sim}.\\%---the first model-theoretic proof of which was given, independently, in \cite{Rot1}.\\

\noindent
\textbf{Corollary \ref {Sim}.}
The following are equivalent for any ring $R$.
\vspace{-.1em}
\begin{enumerate}[\rm (i)]
 \item $R$ is left pure-semisimple, i.e., every left $R$-module is pure-projective.
 \item Every (left) $R$-module is $\Sigma$-pure-injective (i.e., totally transcendental).
 \item Every  (left) $R$-module is Mittag-Leffler. \\
\end{enumerate}

To set the stage, we start the paper  with a brief discussion of the model-theoretic implications around "categorical $\implies$ saturated $\implies$ universal" in the context of an arbitrary variety $\cal V$, \S\ref{V}, and juxtapose them  with the corresponding properties of $R$ when ${\cal V} = \RMod$, the variety of all left $R$-modules, \S\ref{modules}. The last item in the sequence of these equivalences is aforementioned Thm.\ \ref{perfect}.

The current work includes and supercedes  \cite{PiRo}.

\section{Preliminaries}
\subsection{General} Throughout, $L$ is  an elementary (i.e., finitary first-order) language of arbitrary size.  $M \equiv N$ means that the $L$-structures $M$ and $N$ are \texttt{elementarily equivalent}, i.e., satisfy the same $L$-sentences. A class  $\cal K$ is \texttt{elementarily closed} if $M \equiv N\in \cal K$ implies $M\in\cal K$.  An \texttt{elementary class} is a class of $L$-structures axiomatized by a set of $L$-sentences. An $L$-theory is \texttt{complete} iff all of its models are elementarily equivalent. 

$M \prec N$ means that $N$ is an \texttt{elementary extension}  of $M$, i.e., $(M, m)_{a\in M} \equiv (N, m)_{a\in M}$, that is elementary equivalence in the language $L$ with constants (names) adjoined for all elements in $M$.

A structure $M$ is \texttt{saturated}, resp., \texttt{weakly saturated}, if all types in finitely many variables over subsets of $M$ of cardinality $<|M|$, resp., over the empty set,  are realized.  %We sometime use ordinal notation also for cardinals, especially for $

 $\kappa, \lambda, \ldots$ always stand for cardinal numbers. We use  $\omega$ and $\aleph_0$ interchangeably.  
 By \texttt{higher power} or \texttt{higher cardinal} we mean a cardinal $>|L|$. 
 
 Given a class $\cal K$ and a (possibly finite) cardinal $\kappa$, denote by $\cal K^{\geq\kappa}$, $\cal K^{\leq\kappa}$, or $\cal K^{=\kappa}$ the class of members  of $\cal K$ of size $\geq\kappa$, ${\leq\kappa}$, or exactly ${\kappa}$, resp. When $\kappa=\aleph_0$,  instead 
of $\cal K^{\geq \kappa}$ we sometimes write $\cal K^\infty$, the class of infinite members of  $\cal K$. The class $\cal K$ is called \texttt{categorical in $\kappa$} or \texttt{$\kappa$-categorical}, if all members of $\cal K^{=\kappa}$ are isomorphic. 

For \texttt{free} structures  (in a class $\cal K$) consult any universal algebra text or \cite[\S 9.2, p.425]{H}. In the context of a given variety $\cal V$, we let $\cal F$ denote the class of all free structures in $\cal V$. Note that this class is (not empty and) categorical in all higher powers for the simple reason that, above $|L|$, the rank of a free structure is equal to its cardinality.

 \subsection{Modules.} 
 \emph{Module} means (unitary) left module (unless stated otherwise) over an associative ring $R$ with $1$.  Our overall references are \cite{P1} and  \cite{P2}.
 
 A module is  \texttt{free}  if it is of the form $_RR^{(I)}$ for some set $I$. The cardinality of $I$ is the rank of that module. The class of free (left) $R$-modules is denoted by $\RFree$ or  $\cal F$ again when the context is clear. 

 We abbreviate \emph{finitely generated} by \texttt{fg} and \emph{finitely presented} by \texttt{fp}. $\RMod$ is the category of left $R$-modules, and $\Rmod$ the full subcategory of fp objects. Given a class $\A$ of modules, $\Add \A$ denotes the closure of $\A$ under direct summands and arbitrary direct  sums, while  $\add \A$ is used when only finite direct sums are allowed. Further, $\lim\A$ stands for the class of all direct limits (=colimits) of members of $\A$.
 Given a module $M$ and a set $I$, $M^I$  denotes the direct power (= product of $|I|$ copies of $M$), while $M^{(I)}$  stands for the weak direct power (=coproduct or  direct sum of $|I|$ copies of $M$). Thus,  $|M^I|=|M|^{|I|}$ and if $M$ and $I$ are infinite,  $|M^{(I)}|=|M|+|I|$.
  
\textbf{The language.} When it is about $R$-modules, $L$ is the customary language   of $R$-modules with $0, +$ and unary function symbols for every ring element, as in \cite{P1} or \cite{P2}. We are going to recall a few pertinent facts about the use of this language and refer the reader to these two text for any more detail.

 First of all, $|L| =   |R|+\aleph_0$, so for modules \emph{higher power}  means $> |R|+\aleph_0$. The most important syntactic object in this language is what is called a \texttt{pp} (or \texttt{positive primitive}) formula, which is an existentially quantified finite conjunction of $R$-linear equations. We mostly need only unary pp formulas, i.e., with one free variable. Let $\phi(x)$ be such a formula. Then $\phi(M)$, the solution set of $\phi$ in $M$, is the projection of the solution set of the original finite system of equations onto that one free variable place. Being additive, $\phi(M)$ therefore forms a subgroup of the additive group of $M$, a so-called \texttt{pp subgroup} of $M$. Every unary pp formula $\phi$ can be thought of as a functor  into abelian groups: $\phi: \RMod \to \ZMod$. Note that these are the same as Zimmermann's finitary p-functors of \cite{Zim}, whose images are his \emph{endlich matrizielle Untergruppen}, i.e.\, finitary matrix subgroups, which is thus synonymous with what we call pp subgroups. Cf.\,also \cite[\S 1.5]{Fa}.
 
 Every $L$-formula is, uniformly in $\RMod$, equivalent to a boolean combination of pp formulas (in the same variables), a fact known as \texttt{pp-elimination}. It readily implies the following.
 
\begin{fact}
 If $I$ and $J$ are infinite, $M^I\equiv M^{(I)}\equiv M^{(J)}\equiv M^J$, for any module $M$. In particular,  $_RR^I\in\cal M$. (Further, if $R$ is infinite, then $_RR^R\in\cal M$.) 
\end{fact}

The $L$-theory of free modules of infinite rank is denoted by $T$. By  this  fact, all free modules of infinite rank are elementarily equivalent, whence $T$ is complete, cf.\, \cite[Prop.\,7]{SE}. %(Note, this completeness can be traced back to \cite{SE}, which was written before pp elimination was known:
 %\cite[Thm.\,9 and Postscript 3]{SE} says that $T$ is complete and that the product $R^\kappa$ is a model of $T$ whenever $\kappa$ is infinite; cf.\ \cite[Prop.\,17]{SE} for abelian groups.)

\textbf{Free realizations.} Recall that a module is finitely presented iff it is finitely generated and finitely related. The relations, for any given f.p.\ module, can be expressed as a finite system of $R$-linear equations (a conjunction of atomic formulas) that the generators must satisfy. But it is not hard to see that one can do something similar for any tuple in a f.p.\ module if one allows arbitrary pp formulas instead. This  is handily expressed by the following concept, cf.\,Prest's texts. A \texttt{free realization} of a pp formula $\phi(\bar{x})$ is a pair $(A, \bar a)$ with   $A$ a f.p.\ module and a tuple $\bar a$ of the same length as $\bar x$ therein such that (a) $\bar a\in\phi(A)$ and (b) if $B$ is \emph{any} module and $\bar b\in\phi(B)$, then there is a homomorphism $A\to B$ sending $\bar a \to \bar b$, for which we conveniently write $(A, \bar a) \to (B, \bar b)$. Every pp formula has such a free realization and, conversely, every such pair  $(A, \bar a)$ is the free realization of some such pp formula.

\textbf{Purity.} A submodule $M$ of $N$ is called \texttt{pure} (in $N$) if $\phi(N)\cap M\subseteq \phi(M)$ (and hence equal). An embedding or monomorphism is called \texttt{pure} if it is an isomorphism onto a pure submodule. Pure embeddings are like elementary embeddings with only taking pp formulas into account. But by pp elimination one has:

\begin{fact}
 Elementary embeddings of modules are precisely the pure embeddings of elementarily equivalent modules.
\end{fact}

An epimorphism is \texttt{pure} if its kernel is a pure submodule. So, in a short exact sequence, the monomorphim is pure if and only if the epimorphism is. Such sequences are called \texttt{pure-exact}. The pure-exact sequences are precisely the direct limits of split-exact sequences (which hints at the importance of purity).

A module $M$ is called \texttt{pure-injective}  if every pure-exact sequence $0\to M\to N \to P\to 0$ splits. It is interesting to note that pure injectivity is equivalent to what is known as \texttt{algebraic compactness} or \texttt{atomic compactness}, for which see \cite[\S 10.7]{H}. 

A module $P$ is called \texttt{pure-projective} if every pure-exact sequence $0\to M\to N \to P\to 0$ splits.
Since every module is a pure-epic image of a direct sum of f.p.\ modules, a module is pure-projective if and only if it is a direct summand of a direct sum of f.p.\ modules.

\textbf{Projectivity and flatness.}  We obtain the classical notions of \texttt{injective} and \texttt{projective} modules if we require the above splitting for \emph{all} short exact sequences. Correspondingly, as every module is an epimorphic image of a free module, a module is projective if and only if it is a direct summand of a free module. 
 Denote by $\RProj$ or simply $\cal P$ the class of projective (left) $R$-modules. 
 
Closing the class $\cal P=\RProj$ off under direct limit, we get the class $\RFlat$ (also denoted $_R\flat$) of  \texttt{flat} modules: $\lim \cal F\subseteq  \lim\cal P = \RFlat = \lim(\RFlat)$. By Lazard's Theorem, which says that every flat module is a direct limit of free modules of finite rank, all these are equal. %While this can be taken as a definition of flatness, we 
 While  a module is projective iff every short exact sequence ending in it splits, a module is flat iff every short exact sequence ending in it is pure. (Note, split-exact sequences are pure-exact, but not conversely.) %The class of all flat (left) $R$-modules s denoted by $\RFlat$ or simply $_R\flat$.
 See  \cite{L}, \cite{P1}, and \cite{P2} for more detail.
 
 By Kaplansky's Theorem, every (pure-) projective is a direct sum of countably generated (pure-) projectives.
 
 \textbf{Mittag-Leffler modules.} Just as flat modules have better closure properties than projectives, there is an important class containing all pure-projectives with better properties than these. They are the \texttt{Mittag-Leffler} modules, which are the modules that are covered by a system of pure-projective pure submodules. These can, as a matter of fact, be taken countably generated, which shows that countably generated Mittag-Leffler modules \emph{are} pure-projective, and so are countably generated pure submodules of a Mittag-Leffler module.
 A special case is: countably generated pure submodules of pure-projectives are pure-projective---without the countable generation all one can say is that they are Mittag-Leffler.
 Because of  flat+pure-projective=projective, we have  the same in the flat case: countably generated pure submodules of projectives are projective. However,
 flat Mittag-Leffler modules need not be pure-projective. 
  
 Mittag-Leffler modules  were introduced in \cite{RG} and given a model-theoretic treatment in \cite{Rot1} (see also \cite{Rot2} or \cite{P2}), the main result of which is the following  characterization (which can be taken as a definition). A module is Mittag-Leffler if and only if 
 every tuple  realizes a \texttt{finitely generated} pp type. What this means for a module $M$ is this:  for every tuple $\bar a$ in $M$ there is a pp formula $\phi(\bar x)$ (of same arity) which implies all  pp formulas that $\bar a$ satisfies in $M$.  Beware though: implication is intended in the entire module category, i.e. given a pp formula  $\psi$ that $\bar a$ satisfies in $M$, we require that  $\phi$ imply $\psi$ in \emph{every} module, i.e., $\phi(N)\subseteq\psi(N)$ for every $N\in\RMod$. We  write $\phi\leq\psi$ for this. Finite generation means that there be one pp formula in the type so implying the entire pp type. This characterization entails many of the properties Mittag-Leffler modules were shown to have in \cite{RG}. In particular, trivially pure submodules of Mittag-Leffler modules are Mittag-Leffler. Besides, it is  the key to the simple proof of the Bass--Bj\"ork result in \cite{PR} to be discussed  and generalized below.
 
%%Recall that a module is \emph{pure-projective} if it is projective with respect to pure-exact sequences. 
%Pure-projective modules are Mittag-Leffler,   and for countably generated modules the converse is true. Thus countably generated pure submodules of Mittag-Leffler modules are pure-projective, cf.\ \cite{RG} (or the sources above).  (As mentioned in the previous section, 

 \textbf{Stability and $\Sigma$-pure-injectivity.} A module is \texttt{$\Sigma$-pure-injective} if some/all of its infinite weak powers are pure-injective. This is known to be the case if and only if the module (equivalently, any of its weak direct powers) satisfies the dcc on all pp subgroups. (Besides the usual, see also \cite[\S 1.7]{Fa}.)
 
 There is the model-theoretic concept of \texttt{totally transcendental} complete theory (and, by extension, structure), which plays an important role  in general stability theory. We mention only its characterization for modules. (The complete theory of) a  module  is totally transcendental if and only if it is $\Sigma$-pure-injective. This makes sense, as the dcc on pp subgroups is certainly preserved under elementary equivalence. The following fact is the only stability-theoretic result we use (and in fact only in modules, and only once, in Cor.\,\ref{R3}).
 
\begin{fact}\label{tt}(See e.g.\,\cite[10.2.7+8]{H}.)
Totally transcendental theories have saturated models in all higher cardinalities.
\end{fact}

Another important stability-theoretic concept is that of \texttt{superstability} of a complete theory. Since we need this only for modules, and only briefly so, note that a module $M$ is \texttt{superstable} if and only if it has the dcc on pp subgroups with infinite factors, i.e., iff there is no infinite descending chain $\phi_0(M)\supset \phi_1(M)\supset \phi_2(M)\supset \ldots$ with all $\phi_i(M)/\phi_{i+1}(M)$ infinite. This is less familiar a condition algebraically, however, the following is easy to see, cf.\, \cite{P1} or \cite{H}.

\begin{fact}\label{sstt}
 An infinite power $M^{(\omega)}$ of a module is superstable iff it (equivalently, $M$)  is totally transcendental (that is, $\Sigma$-pure-injective).
\end{fact}

 \textbf{Definable subcategories}---an unfortunate term that sticks---are (full subcategories of $\RMod$ on)  classes of modules closed under product, direct limit and pure substructure. These are precisely the classes of modules that are axiomatized by implications of pp formulas, \cite[Thm.\,3.4.7]{P2}. An intersection of a set of those is a definable subcategory again, which allows us to define the definable subcategory  \texttt{generated by a class} $\cal K$ as the smallest  definable subcategory (of $\RMod$) containing $\cal K$.
 
 \subsection{Rings} $R$ is always an associate ring with $1$ and Jacobson radical $J$. Following H.\,Bass, $R$ is \texttt{left perfect} if every flat (left) $R$-module is projective. Bass proved that this is equivalent to the dcc on principal \emph{right} ideals, while  Bj\"ork extended this by showing that left perfect rings have the dcc even on all fg right ideals.
 
 Following C.\,U.\,Chase, $R$ is \texttt{right coherent} if all direct powers $_RR^\kappa$ are flat. 
It is implicit in Chase's proof (made explicit in \cite[Prop.\,1.24(ii)]{Rot1} that it suffices to check $\kappa =  |R|+\aleph_0 =|L|$, but, beware!, not necessarily $\aleph_0$, \cite[Exple.\,4.2]{Len}. (What if $I$ is finite? Then $R^I = R^{(I)}$ is always flat anyway.) 
 
 $R$\/ is a \texttt{left $\Sigma$-pure-injective} ring if $_RR$ is a $\Sigma$-pure-injective module. By the dcc on pp subgroups (which are right ideals), such rings are left perfect. 
 It is well known that $R$ is right coherent iff all pp defined right ideals are fg.  Conversely, left perfect and right coherent rings are, by the Bass-B\"ork Theorem, $\Sigma$-pure-injective.
 
 See \cite{P1} and \cite{P2} for all this.

%\cite[Cor.\,2]{SE} says COROLLARY 2. If A is a left perfect ring, there exists (up to isomorphism) exactly one projective indecomposable module if and only if A/J is simple (as a ring).

\begin{fact}\label{elem} \cite[Thms.\,5, 6]{SE} (Cf.\,\cite{P1}, \cite{P2}.)
\begin{enumerate}[\rm (1)]
 \item The class $_R\flat=\RFlat$ is elementary (equivalently, elementarily closed) if and only if $R$ is right coherent. 
\item The class $\cal P=\RProj$ is elementary (equivalently, elementarily closed)  if and only if $R$ is left perfect and right coherent.
\item The class $\cal F=\RFree$ is elementary (equivalently, elementarily closed)  if and only if $R$ is either local right artinian or finite with $R/J$ a simple ring. 
\end{enumerate}
\end{fact}

%\newpage
\subsection{Decomposition theory of projectives over perfect rings}\label{dec}  Let $\cal K$ be a class of $R$-modules \texttt{with de\-com\-po\-si\-tion theory}, by which we mean that $\cal K$ possesses a set (sic!), $\Ind(\cal K)$, of \texttt{indecomposables} (i.e., modules that cannot be decomposed into a direct sum of two proper submodules) with the following two properties.  First of all,  every direct sum of members of $\Ind(\cal K)$ is in $\cal K$ and second, and more importantly, every member of $\cal K$ can be decomposed into a direct sum of members of $\Ind(\cal K)$ in an essentially unique way, which means up to permutation and isomorphism of the indecomposables involved. 

Two examples play a role here. One is the well-known classical example of semisimple artinian rings (i.e., an artinian ring with Jacobson radical zero). The other is that of projective modules over a left (semi-)perfect ring, see \cite[Thm.\,1]{SE} or \cite[Thm.3.10]{Fa} (for semiperfect rings). The former is in fact a special case of the latter, as $\cal K=\RMod=\RProj$, in which case $\Ind(\cal K)$ is the (finite!) set of simple modules. In the latter case,  $\Ind(\RProj)$ is also a finite set, for the simple reason that $_RR$ itself is projective and must therefore so decompose (and $1\in R$ and every $P_i$ has to occur, by uniqueness). So we can write $\Ind(\RProj)=\{P_i\,:\,i<n\}$. Then $_RR = \bigoplus_{i<n} P_i^{(r_i)}$ for some integers $r_i$ ($i<n$). For purposes of this paper only, denote the integer $n$, i.e., the number of non-isomorphic projective indecomposables (left) $R$-modules, by $\n(R)$ and call $\sum_{i<n} r_i$ the \texttt{degree} of $R$, denoted $\deg(R)$. Note that $|P_i|\leq|R|\leq |L|$.

Further, if $n(R)>1$ and $\kappa, |L|\leq \lambda$, then, by uniqueness of decompositions, $P_0^{(\kappa)} \oplus P_1^{(\lambda)}$ and $P_0^{(\lambda)} \oplus P_1^{(\kappa)}$ are two non-isomorphic projectives of power $\lambda$ (and $P_0^{(\lambda)} \oplus P_1^{(\lambda)}$ is yet another one)). Clearly the first two of them are not free. In fact, a direct sum $\bigoplus_{i<n} P_i^{(\kappa_i)}$ with at least one of the $\kappa_i$ infinite is free if and only if all the $\kappa_i$ are the same. Also, it is well-known that $\bigoplus_{i<n} P_i^{(\kappa_i)}$ with all $\kappa_i$ infinite is elementarily equivalent to, say, $R^{(\kappa_0)}$, hence a model of $T$. Notice its cardinality in higher powers: if at least one of the $\kappa_i\geq |L|$, then $|\bigoplus_{i<n} P_i^{(\kappa_i)}| = \max\{\kappa_i\,:\, i<n\}$.  Also,  if $R$ is left semiperfect, then $_RR$ is indecomposable if and only if $\n(R)=\deg(R)=1$.
Let us summarize (and add (iv), (3), and (4) from \cite{SE}).

\begin{lem}\label{n=1}  \cite{SE} Suppose $R$ is left perfect.
\begin{enumerate}[\rm (1)]
 \item 
\begin{enumerate}[\rm (a)]
\item $\n (R)=1$ if and only if $\cal P$ is categorical in a higher power, i.e., all projectives of higher power are free.
\item This is the case if $\cal M$ is categorical in a higher power, i.e., all models of $T$ of higher power are free, but not conversely, see Thm.\,\ref{diamond} below.
\end{enumerate}
\item \cite[Cor.\,3 and 5]{SE}. The following are equivalent.
\begin{enumerate}[\rm (i)]
\item $\n (R)=\deg(R)=1$.
\item All projectives are free, i.e., $\cal P \subseteq\cal F$.
\item $_RR$ is indecomposable.
\item $R$ is local.
\end{enumerate}
\item \cite[Prop.\,3]{SE}. Local left perfect and right coherent rings are right artinian.
\item \cite[Cor.\,2]{SE}. $R/J$ simple iff $n(R)=1$.
\qed
\end{enumerate}
\end{lem}

In \cite[Thm.\,6]{SE} this kind of argument is used to prove that $n(R)$ must be $1$ when the class of free modules is elementary. Incidentally, \cite{SE} was written when there was apparently no text book treatment of the projective decomposition theory (not even for perfect rings), which seems to go back to, independently, G.\,B.\,Klatt (Ph.D. unpublished, 1969) and B.\,J.\,Müller (1970), cf.\,\cite[Notes Ch.\,3]{Fa}.

\section{From categorical to universal}
\subsection{Varieties in general}\label{V} Let $\cal V$ be a variety in the sense of universal algebra (i.e., an equational class) in a language $L$. %By \texttt{higher cardinal} or \texttt{higher power} we mean  any cardial $>|L|$.  
A natural class to look at is $\cal V^{>|L|}$, the class of all members of $\cal V$ in a higher power, as it contains  exactly one free algebra of every cardinality $\kappa>|L|$ (whose rank is also $\kappa$).

 The smallest subclass of $\cal V$ considered in this paper is $\cal F$, the class of all free structures in $\cal V$. Denote by $F_\kappa$ the free structure in $\cal V$ of rank $\kappa$. 
Note that $\cal F$ is categorical in all higher powers, since for $\kappa>|L|$, the rank $\kappa$ is the cardinality $|F_\kappa|$. 
If $\cal F\subseteq \cal W\subseteq\cal V$, the categoricity of $\cal W$  in some $\kappa>|L|$ is, for mere reasons of cardinality of free structures, equivalent to the property: \\

\emph{Every structure in $\cal W^{=\kappa}$ is free.}\\

Thus, classical questions like "when is every projective, resp., flat module of certain cardinality free?" can be viewed as a categoricity question for such classes $\cal W$, where here the ambient variety is $\RMod$, that of all left $R$-modules. For investigations of categoricity of entire (quasi-) varieties, see \cite{Kea}.

Back to an arbitrary variety $\cal V$, let $T:=\Th({\cal F}^{>|L|})$, the $L$-theory of its free members of infinite rank. As in modules, also in this more general case $T$ is complete, for all free $L$-structures of infinite rank are elementarily equivalent---by back-and-forth, see e.g.\, \cite[Exercise\,9.2.8]{H}. %(This is a classical observation, and immediate consequence of the L\"owenheim-Skolem Theorem.) 
Let ${\cal M} =\Mod T$, the class of all structures elementarily equivalent to a free structure in $\cal V$ of infinite rank.

Consider the following statements that are arranged in decreasing order of strength.\\ 

% https://tikzcd.yichuanshen.de/#N4Igdg9gJgpgziAXAbVABwnAlgFyxMJZARgBoAGAXVJADcBDAGwFcYkQAKAJgEoQBfUuky58hFOVLFqdJq3YdifQcOx4CRLlJkMWbRJ14B9YAB1TAYyYACALL8BQkBjViiZLjrn7OSk+atGOwcVZxF1cWQAZm0aXXkDDijlJxdRDRQAVljZPQUAFhTVdMj8ii88xOTzADMAJ3oLYGJ+YC4Q1PC3FAA2HPifDkzlGRgoAHN4IlB6iABbJEkQHAgkLVyEsGZGRhpGegAjGEYABS6MkEYYGpwQPawwHygIZgOrxxm6+aQyZdXEJYDJBbHZ7Q7HM6uC5XG53S4PJ4vN5sUKzBaIGJ-Nb3R7sZ6vd40AAWMHoUGB212y3oWEY7EguNRX3RvxWSExjAReKRhJAJLJFNB1Np9IIKKcaOxWMQZXhuIM+ORcP55MQIKpOBpdIMDPFn2+iGy0oA7DjEQS2MTSar1TRNSKdWKPiBJTK7f8jZz5SBFbyVYKNVrRYyJcykH1pZ6uQqeZa+daA3ag46Q-r0bK2YgI17ETgcGNna6jZnTXKfLrlQmU3H7drwE6wUdTudxCAHthYIWw4b3UhSznuRbKwK1ZSkw764zKPwgA
\begin{tikzcd}
                                                  & (2) \arrow[rd, Rightarrow]          &                                    &                                               & (3)\frac{1}{2} \arrow[rd, Rightarrow] &                                                                    &     \\
(1) \arrow[ru, Rightarrow] \arrow[rd, Rightarrow] &                                     & (2)_{\cal M} \arrow[r, Rightarrow] & (3) \arrow[rr, Rightarrow] \arrow[ru, dotted] &                                       & (4)_{\cal N} \arrow[r, Rightarrow] \arrow[r, phantom] \arrow[r, Rightarrow] & (5)_{\cal N} \\
                                                  & (1)_{\cal M} \arrow[ru, Rightarrow] &                                    &                                               &                                       &                                                                    &    
\end{tikzcd}

%% https://tikzcd.yichuanshen.de/#N4Igdg9gJgpgziAXAbVABwnAlgFyxMJZARgBoAGAXVJADcBDAGwFcYkQAKAJgEoQBfUuky58hFOVLFqdJq3YdifQcOx4CRLlJkMWbRJ14B9YAB1TAYyYACALL8BQkBjViiZLjrn7OSk+atGOwcVZxF1cWQAZm0aXXkDDijlJxdRDRQAVljZPQUAFhTVdMj8ii88xOTzADMAJ3oLYGJ+YC4Q1PC3FAA2HPifDkzlGRgoAHN4IlB6iABbJEkQHAgkLVyEsGZGRhpGegAjGEYABS6MkEYYGpwQPawwHygIZgOrxxm6+aQyZdXEJYDJBbHZ7Q7HM6uC5XG53S4PJ4vN5sUKzBaIdYrJAxeGPdjPV7vGgACxg9CgwO2u2W9CwjHYkDxqK+6N+WMQOMYCPxSKJIFJ5MpoJpdIZBBRTjRaxo7LKuMRhLYJLJFMQIOpOFp9IMjIln2+iGyfyQAHZ7niDATkXCBar1TKtWKmZKWUg5eyjVyLSArXzbUKNY6deKPiApYg+sbDeaFdblYK1VSHaLg879ej3f9I17ETgcGMBJR+EA
%\noindent
%\begin{tikzcd}
%                                                  & (2) \arrow[rd, Rightarrow] &                                                           &                                               & (3)\frac{1}{2} \arrow[rd, Rightarrow] &                           &     \\
%(1) \arrow[ru, Rightarrow] \arrow[rd, Rightarrow] &                            & (2)_{\cal M} \arrow[ru, Rightarrow] \arrow[r, Rightarrow] & (3) \arrow[rr, Rightarrow] \arrow[ru, dotted] &                                       & (4) \arrow[r, Rightarrow] & (5) \\
%                                                  & (1)_{\cal M}               &                                                           &                                               &                                       &                           &    
%\end{tikzcd}

\begin{enumerate}
 \item\hspace{.5em} All members of $\cal V$ are free.
 \setcounter{enumi}{0} 
  \item\hspace{-.6em}$_{\cal M}$ All members of $\cal M$ are free.
 % \item $\cal V$ is categorical in all cardinals that a member of $\cal V$ may have. Not true! Not for $\aleph_0$.
 \item\hspace{.5em}   All members of $\cal V^{>|L|}$ are free (equivalently, $\cal V$ is categorical in all/some higher powers).
  \setcounter{enumi}{1} 
  \item\hspace{-.6em}$_{\cal M}$\label{five}  All members of $\cal M^{>|L|}$ are free (equivalently, $\cal M$ is categorical in all/some higher powers).%$\cal V$ is categorical (equivalently, categorical in \emph{all} $\kappa>|L|$) 
  %\item $\cal V^\infty$ is categorical (and complete).    %infinite cardinals. Not true! Not for $\aleph_0$.
% \item$T^\infty$ is categorical.
\item\hspace{.5em} \label{six} The free structures in $\cal V$ of rank $\kappa>|L|$ are saturated.
\setcounter{enumi}{2} 
\item\hspace{-2mm}$\frac{1}{2}$ The free structures in $\cal V$ of rank $\kappa>|L|$ are totally transcendental and $\preceq$-universal in $\cal M$.
 \item\hspace{-.6em}$_{\cal N}$ \label{seven} Given an intermediate  class  $\cal F\subseteq \cal N\subseteq\cal M$,
 the free structures in $\cal V$ of rank $\kappa>|L|$ are  $\preceq$-universal in $\cal N$ and also superstable.
%The free structures in $\cal V$ of rank $\kappa>|L|$ are superstable and $\preceq$-universal in $\cal M$.
 \item\hspace{-.6em}$_{\cal N}$ \label{eight} Given an intermediate  class  $\cal F\subseteq \cal N\subseteq\cal M$,
 the free structures in $\cal V$ of rank $\kappa>|L|$ are  $\preceq$-universal in $\cal N$.
 %The free structures in $\cal V$ of rank $\kappa>|L|$ are pure-universal in $\cal M$.
  \end{enumerate}
 
The implications in the diamond and $(\ref{seven})\implies(\ref{eight})$ are trivial. (Condition (\ref{eight}) will be useful  in the module case below for the specific choice of $\cal W=\RFlat$.) $(\ref{five})_{\cal M}\implies(\ref{six})$ is a standard model-theoretic fact \cite[Cor.\,12.2.13]{H}. So is that saturation implies $\preceq$-universality (within $\cal M$) \cite[Thm.\,10.1.6]{H}. The state of affairs of  $(\ref{six})\implies(\ref{seven})$ is this: 

(\ref{six})$\implies$(\ref{six})$\frac{1}{2}$ holds for arbitrary varieties in a countable language  \cite[Lem.\,1]{B-S} and for the variety of all modules over arbitrary (not necessarily countable) rings \cite[Cor.\,3.5(b)]{KP}. 

\begin{quest}
Does (\ref{six})$\implies$(\ref{six})$\frac{1}{2}$ hold  in general? 
\end{quest}

(\ref{six})$\frac{1}{2}\hspace{-2mm}\implies\hspace{-2mm}(\ref{seven})$  is always true, simply because  total transcendence implies superstability \cite[Thm.6.7.5]{H} (which is obvious  in modules, see Fact \ref{sstt} above).

Tarski asked,  whether (1)$\implies$(2) was reversible,  see \cite{Baldwin-Lachlan}, where a counterexample was given. We present a simple example of modules in Rem.\,\ref{rev}(\ref{Tarski}) below. 

\begin{rem}
Note that free groups are never saturated. Neither are the free abelian groups: $\mathbb{Z}^{(I)}$ is saturated for no set $I$, but $\mathbb{Z}^{(I)}\oplus \mathbb{Q}^{(I)}$ is, provided $I$ is infinite. (For finite $I$, the group $\mathbb{Z}^{(I)}\oplus \mathbb{Q}^{(J)}$ is saturated for every infinite $J$.) 
Note that  $\mathbb{Z}$ is not (left) perfect, nevertheless we have a decomposition theory for $\cal P = \cal F$ (in this case) with $\n(\mathbb{Z})=\deg\mathbb{Z})=1$.
\end{rem}
%We also have (\ref{six})$\frac{1}{2}\implies$(\ref{seven}), where

\subsection{Modules}\label{modules} Fact \ref{sstt} applied to
 free modules of infinite rank proves at once the module case of  (\ref{six})$\frac{1}{2}\Longleftrightarrow$\,(\ref{seven}) (and thus (\ref{six})$\implies$(\ref{six})$\frac{1}{2}$) as noted in \cite[Cor.\,3.5(b)]{KP}. Here is another proof that  yields a more general result.

\begin{prop}\label{weaksat} Let $M$ be a  Mittag-Leffler module whose power $M^{(\omega)}$ is weakly saturated. Then $M$ and $M^{(\omega)}$ are totally transcendental. 
In particular, {\rm (\ref{six})$\implies$(\ref{six})$\frac{1}{2}$} holds for modules.
\end{prop}
\begin{proof}
 Since free modules are Mittag-Leffler, the second clause is a special case of the first, where $M= {_RR}$.
 For the first, %consider a Mittag-Leffler module $M$ which is saturated. 
 suppose  $\phi_0(M)\supseteq \phi_1(M)\supseteq \phi_2(M)\supseteq \ldots$ is an infinite  descending chain of pp subgroups   in $M$. Consider the type\\ $p:=\{\phi | \phi\, \text{is pp and}\,\phi_i\leq \phi\, \text{for some}\, i<\omega\}\cup\{\neg\psi | \psi\, \text{is pp and}\,\psi\leq\phi_i \,\text{for all}\, i<\omega\}$.  

If $p$ is consistent with  $M^{(\omega)}$, it is realized by an element $a\in M^{(\omega)}$. As the module is Mittag-Leffler, $a$ must satisfy a pp formula $\psi$ with $\psi\leq\phi_i$ for all $i$. But then $p$ contains $\neg\psi$, which $a$ must also satisfy, a contradiction,

So $p$ is inconsistent with  $M^{(\omega)}$. Hence there are $j<\omega$ and finitely many $\psi_k$ below all $\phi_i$ (wrt the entire category $\RMod$) such that $\phi_j\leq_{M^{(\omega)}} \bigvee_{k<n} \psi_k$. By McKinsey's Lemma \cite[Cor.\,9.1.7]{H},  $\phi_j\leq_{M^{(\omega)}}  \psi_k$ for some $k<n$. But $\psi_k\leq \phi_i$ for all $i<\omega$, hence $\phi_j\leq_{M^{(\omega)}} \phi_i$ for all $i<\omega$, which proves that the chain must stabilize in $M^{(\omega)}$, as desired.
 %Any realization of $p$ has pp type $p$, which is realized in $M$, for then    $p$ is consistent with $M$
\end{proof}

%As in modules 
%(\ref{six})$\frac{1}{2}$ and (\ref{eight}) being equivalent, w
We now have the following simplified picture for $\RMod$.\\

% https://tikzcd.yichuanshen.de/#N4Igdg9gJgpgziAXAbVABwnAlgFyxMJZARgBoAGAXVJADcBDAGwFcYkQAKAJgEoQBfUuky58hFOVLFqdJq3YdifQcOx4CRLlJkMWbRJ14B9YAB1TAYyYACALL8BQkBjViiZLjrn7OSk+atGOwcVZxF1cWQAZm0aXXkDDijlJxdRDRQAFljZPQVMlNV0yIBWHPifDhLlGRgoAHN4IlAAMwAnCABbJEkQHAgkLVyEsGZGRhpGegAjGEYABXC3A0YYFpwQSawwHygIZmnVx1aO7sQyPoHEXoqkUfHJmbnF1wyQVfXN9+3d-cO2ULtLqDGj9JAxb47dh7A5HGgACxg9CgdzGEz69CwjHYkChgNOSAuYMQEMYP2hfzhIERyNRDwxWJxBABTiBZyGxOykN+sLYCKRKMQ93ROEx2IMuJZJ2BiC5xLK3IpvK+NMFwtBYqZeNZBMQCuJADYtlCDDD-iqBXSRZqJcyBJR+EA
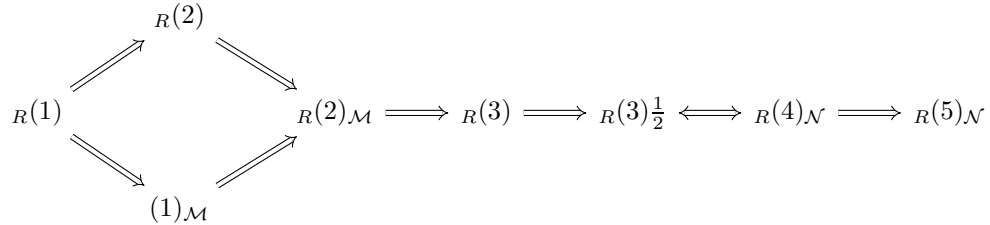
\noindent
\begin{tikzcd}
                                                  & _R(2) \arrow[rd, Rightarrow] &                                                           &             &              &                           &     \\
_R(1) \arrow[ru, Rightarrow] \arrow[rd, Rightarrow] &                            & _R(2)_{\cal M}  \arrow[r, Rightarrow] & _R(3) \arrow[r, Rightarrow] & _R(3)\frac{1}{2}  \arrow[r, Leftrightarrow] &  _R(4)_{\cal N} \arrow[r, Rightarrow] & _R(5)_{\cal N} \\
                                                  &   (1)_{\cal M} \arrow[ru, Rightarrow]               &                                                           &             &              &                           &    \\
\end{tikzcd}

The left subscript $R$ is to indicate that $_R(n)$ is the statement $(n)$ for ${\cal V} = \RMod$, the variety of left $R$-modules. The next task is to find ring- and  module-theoretic equivalents of these. Here is a list of the ring-theoretic equivalents for guidance, which we are going to prove, together with other equivalents.\\ % to be proved one by one afterwards. These will also make eventually clear that none of the implications are reversible. \\

% (1) through (\ref{eight}) for the case ${\cal V} = \RMod$. We formulate those as $_R(n)$ in such a way that for ${\cal V} = \RMod$, the condition $(n)$ is equivalent to $_R(n)$.

\begin{description}
   \item $_R(1)$\hspace{1em} $R$ is a division ring.
   \item $_R(1)_{\cal M}$ $R$ is right artinian and either local or  finite with $R/J$ a simple ring.
   \item $_R(2)$ \hspace{1em}$R$ is a matrix ring over a division ring.  
   \item $_R(2)_{\cal M}$ $R$ is left perfect and right coherent with $R/J$  a simple ring. 
   \item $_R(3)$\hspace{1em} $R$ is left perfect and right coherent.
   \item $_R(4)$\hspace{1em} $R$ is left-$\Sigma$-pure-injective.
   \item $_R(5)$\hspace{-.1em}$_{\RFlat}$\hspace{1em} $R$ is left perfect.
\end{description}

\begin{rem}\label{rev}
\begin{enumerate}[\rm(a)]
 \item The implications between these are as shown in the graph above. This is obvious, except for possibly the last two. For (4)$\implies$(5) note that left-$\Sigma$-pure injectivity is equivalent to the dcc on all pp subgroups, while left perfectness is equivalent to that on the `subset' of principal right ideals. For (3)$\implies$(4) use the fact that right coherence is equivalent to finite generation of all right ideals which are pp subgroups of $_RR$ and then apply the dcc on finitely generated ideals given by Bj\"ork's strengthening of Bass' Theorem.
 \item None of the  implications are reversible (except as shown). This is clear for the implications in the diamond and $_R(2)_{\cal M}\implies _R(3)$. Every left $\Sigma$-pure-injective ring is left perfect, but there are such that are not right coherent, see \cite[Satz 6.5]{Zim} and also \cite[Example 4.4.23]{P2}, hence $_R(3) \implies _R(4)$ is irreversible as well. (Note, the saturated model of $T$ is not free over this ring.) For the irreversibility of the final implication, \cite[Exple 17]{ZZ} exhibits a ring that is left perfect but not left $\Sigma$-pure injective (totally transcendental).
  
    \item \label{Tarski} Recall, the implication (2)$\implies$(1) about general varieties was conjectured by Tarski. Once we see, in Thm.\,\ref{diamond}, that these are equivalent to the corresponding ring-theoretic statements listed above, it is immediate that the variety of modules over a ring $R$ of, say $2\times 2$ matrices over an infinite field $k$ is a very natural counterexample, for one has $\n(R)=1$ and $\deg(R)=2$.

\end{enumerate}

\end{rem}

We start with  the statements  in  the diamond. Some of them are known, cf.\, Lemma \ref{n=1}. All statements involving $T$ or $\cal M$ are new (or from \cite{PiRo}).%, like the equivalence of (ii) and (v) in (3), which is Prop.\,2.11, and that of (i), (ii) and (v) in (4), which is part of Prop.\,2.14).

\begin{thm} \label{diamond}
\begin{enumerate}[\rm (1)]
 \item The following are equivalent.
 \begin{enumerate}[\rm (i)]%\label{thm}
 \item {\rm $_R(1)$:} All members of $\cal V=\RMod$  are free, i.e., $\cal V\subseteq \cal F$ (and clearly $\n(R)=\deg(R)=1$).
\item $R$ is a  division ring.
\end{enumerate}
%\begin{enumerate}[\rm (a)]
%\item %If $\cal P$ is categorical in a higher power, i.e., all projectives of higher power are free.
%\item %If $\cal M$ is categorical in a higher power, i.e., all models of $T$ of higher power are free.
%\end{enumerate}
\item The  following are equivalent.
 
\begin{enumerate}[\rm (i)]%\label{thm}
\item  {\rm $_R(2)$:} all modules of higher power are free.
\item $\cal V$ is categorical in higher power.
\item All modules are projective, i.e., $\cal V\subseteq \cal P$, and $\n(R)=1$. 
%\item %$_R(1)_{\cal M}$ $R$ is right artinian and either local or  finite with $R/J$ a simple ring.
%\item The definable subcategory generated by $\RFlat$ is categorical in higher power.
\item $R$ is a matrix ring over a division ring.
\end{enumerate}
%\item The definable subcategory generated by $\RFlat$ is categorical in higher power.
%\item $\cal V$ is categorical in higher power.

%\item $R^{(\kappa)}$ is saturated, for some (every) $\kappa\geq\omega$.
% (i.e.,
\item\label{2M} The  following are equivalent.
 
\begin{enumerate}[\rm (i)]%\label{thm}
\item\label{(2iv)} {\rm $_R(2)_{\cal M}$:}  all members of $\cal M$ of higher power are free.
\item\label{(2iii)}  $T$ is categorical in all (some) higher powers.
\item\label{(2v)}  All models of $T$ are projective, i.e. $\cal M\subseteq \cal P$, and $\n(R)=1$.
\item\label{(2ii)}  The definable subcategory generated by $\RFlat$ is categorical in higher power.
\item\label{(2i)}  $R$ is left perfect and right coherent and $R/J$ is a simple ring.

%\item $R^{(\kappa)}$ is saturated, for some (every) $\kappa\geq\omega$.
% (i.e.,
\end{enumerate}
\item The  following are equivalent.
 
\begin{enumerate}[\rm (i)]%\label{thm}
\item  {\rm $_R(1)_{\cal M}$:} all models of\/ $T$ are free, i.e., $\cal M\subseteq \cal F$.
\item The class $\cal F$ is elementary.
%\item All  are free.
\item  $\cal M\subseteq \cal P$ and $\n(R)=1$, and if $R$ is infinite, then $\deg(R)=1$.

\item Every infinite member of the definable subcategory generated by $\RFlat$ is free.
\item $R$ is right artinian and either local or  finite with $R/J$ a simple ring.
%\item %$_R(1)_{\cal M}$ $R$ is right artinian and either local or  finite with $R/J$ a simple ring.

%\item $R^{(\kappa)}$ is saturated, for some (every) $\kappa\geq\omega$.
% (i.e.,
\end{enumerate}

\end{enumerate}
\end{thm}
\begin{proof}
(1). If all modules are free, the ring cannot have proper (left) ideals. Conversely, every vector space is free.

(2). (i) and (ii) are clearly equivalent and imply (iii), for, given any module $M$, large enough powers of it are free, hence $M$ is a direct summand of a free module and thus projective; and Lemma \ref{n=1} yields $\n(R)=1$. 

If, conversely, (iii) holds, every module is a weak direct power of a single projective indecomposable, whence in any given higher power they are isomorphic, so (i) through (iii) are equivalent.

(iii)$\implies$(iv). Every module is projective iff the ring is semisimple artinian. By the Wedderburn-Artin Theorem, $R$ is a direct product of matrix rings over division rings, however 
$\n (R)=1$ implies that there is only one such, so $R$ is a simple artinian ring, hence a matrix ring over a division ring.

(iv)$\implies$(iii) is well known.

(3).(v)$\implies$(iv). If $R$ is left perfect and right coherent, flat modules are projective and these form an elementary class. Therefore, the definable subcategory generated by $\RFlat$ is $\RFlat=\cal P (=\RProj)$. We claim $\cal P$ is categorical in higher powers. By the decomposition theory of projectives over perfect rings, \S \ref{dec}, every projective is, essentially uniquely, a direct sum of (projective) indecomposables.  But by hypothesis and Lemma \ref{n=1}, there is only one, call it $P$, so every projective is isomorphic to $P^{(I)}$ for some set $I$. Therefore all projectives of power $>|P|$ are isomorphic. It remains to recall that $|P|\leq |R|$ to conclude the proof of the claim.

(iv)$\implies$(iii). Assuming (iv), every flat of higher power is free. In particular, so is every large enough weak power of the model $F\oplus B$ from Lemma \ref{flatmodel}. Then that model is projective, and by that very lemma, $R$ left perfect. As $R^I\equiv R^{(I)}$ for every infinite set $I$, such direct powers $R^I$ are in the definable subcategory generated by the flats and thus free for large enough $I$. Hence $R$ is right coherent. Then flatness=projectivity is an elementary property, hence $\cal M\subseteq \cal P$. But by assumption, $\cal P$ is also categorical in higher powers, hence $\n(R)=1$, by Lemma \ref{n=1}.

(iii)$\implies$(ii). If all models of $T$ are projective with $\n(R)=1$, then, by Lemma \ref{n=1}(1), $T$, resp., $\cal M$ is categorical in higher powers, which clearly implies also (i).

    Finally, assuming (i), we prove (v).  First of all,  all modules $R^I$ are in $\cal M$ (at least for infinite $I$, for then  $R^I\equiv  R^{(I)}\in \cal M$), hence all such modules are flat. This implies right coherence (as flatness for finite $I$ is trivial). Second, the flat model $F\oplus B$ of $T$ from Lemma \ref{flatmodel} below is also free, so $B$ is projective and hence $R$ left perfect. Then Lemma \ref{n=1} applies, whose part (1)(b) yields $\n(R)=1$, hence also, by part (4), that $R/J$ is simple.  
    
    (4).  (i)$\implies$(iii).  If all models of $T$ are free, part (\ref{2M}) applies, whose (v) and (iii) show that $R$ is left perfect and right coherent and that   $\cal M\subseteq \cal P$ and  $\n(R)=1$. Then $_RR=P^{(d)}$ with $d$ an integer, namely, $d=\deg(R)$, and every projective a weak direct power of $P$.
If $R$ is infinite, so is $P$, the unique indecomposable projective. 

We claim $P$ is a model of $T$. Using L\"owenheim-Skolem, choose  an elementary extension $P'$  of $P$ of higher power. By left perfectness and right coherence,  $\RFlat=\cal P$ is an elementary class, whence $P'$ is also projective, hence $P'\cong P^{(I)}$ for some infinite set $I$. As $I$ is infinite (and $d$ finite),  also $P'\cong {_RR}^{(I)}$, so $P'$ is free and thus a model of $T$ Consequently, $P\prec P'$ is a model as well, as claimed.

By (i), $P$ is free, hence $d=\deg(R)=1$.

 (iii)$\implies$(v). First of all, (iii) and part (\ref{2M}) make $R$ left perfect and right coherent  again. As $\n(R)=1$, $R/J$ is simple, which is all we need if $R$ is finite. If it is infinite, then by hypothesis $\deg(R)=1$. Lemma \ref{n=1}(2), $R$ is local. (3) of the same lemma then concludes the proof.
 
  (v)$\implies$(i). If $R$ is artinian, it is  left perfect and right coherent, so part (\ref{2M}) applies and yields $\cal M\subseteq \cal P$ and $\n(R)=1$. By Kaplansky's Theorem, if $R$ is local,  local $\cal P\subseteq \cal F$, so (i) follows for infinite $R$. If now $R$ is finite, we still have $\n(R)=1$ and hence $_RR=P^{(d)}$ for the unique projective indecomposable, but now any infinite projective module is an infinite weak direct power of $P$ and so automatically a  weak direct power of $_RR$, that is free. So every infinite projective is free, and we have (i).
  
  (ii) and  (v)  are equivalent by  Fact \ref{elem}(3).
  
  (iii)$\implies$(iv). By (iii), we have $_RR = P^{(\deg(R))}$ and $\deg(R)=1$ in case  $R$ is infinite. If $\deg(R)=1$, clearly all projectives are free. But by (3), the ring is left perfect and so flats are projective. If $\deg(R)>1$, the ring must be finite. But then all infinite (flats=) projectives are \emph{infinite} weak direct powers of $P$ and hence also free.
  
  Finally, we prove  (iv)$\implies$(iii). Note (4)(iv)$\implies$(3)(iv)  and thus (3)(iii). It remains to see why $\deg(R)=1$ in case  $R$ is infinite. Then $P$ itself must be infinite too. To be free we have to have $\deg(R)=1$.
\end{proof}

\begin{rem}
\begin{enumerate}[\rm(a)]
 \item 
 That `some' implies `all' in (\ref{2M}) above is an instance of a deep theorem, the celebrated Morley-Shelah Theorem (see \cite[Thm.\,12.2.1]{H} for references), however for modules one obtains it without much effort.
 \item By a theorem of Govorov, $R$ is left perfect and local if and only if \emph{all} flat modules are free, cf.\,\cite[Cor.\,5]{SE}.
\end{enumerate} 
\end{rem}

\begin{exam} Part of the proof of   \emph{(4)(i)$\implies$(iii)} shows that if $R$ is an infinite left perfect and right coherent ring with $\n(R)=1$, then $_R\flat=\cal P=\cal M$ and all projective $R$-modules are elementary equivalent, for $P\prec  _RR^{(I)}$ and $P\subseteq Q\subseteq _RR^{(I)}$ yield $Q\equiv P\equiv _RR^{(I)}\in\cal M$, which is true for any projective module $Q$ (in particular for $_RR$).

Note that in this case $\cal F= \cal M$ (i.e., all models of $T$ are free) precisely when also $\deg(R)=1$.
\end{exam}

We turn to $_R(5)$\hspace{-.1em}$_{\cal M\cap {\RFlat}}$ and proceed to the remaining two conditions in two corollaries.  Note that the equivalence of (iii) again rests  on Lemma \ref{flatmodel} further down in \S\ref{Bass1}.

By a defining feature of projective modules, the free module of infinite rank $\kappa$ is pure-universal   for ${\kappa}$-generated projectives (even  direct-summand-universal). If this is true for all flat modules instead, we have a strong ring-theoretic property.

\begin{thm}\label{perfect} 
The  following are equivalent.
 
\begin{enumerate}[\rm (i)]%\label{thm}
\item\label{(Tiii)}  {\rm $_R(5)$\hspace{-.1em}$_{\cal M\cap {\RFlat}}$:} $R^{(\kappa)}$ is $\prec$-universal (equivalently, pure-universal) among the $\kappa$-generated flat models of  $T$, for some (every) $\kappa\geq\omega$. 
\item\label{(Tii)} $R^{(\kappa)}$ is pure-universal among $\kappa$-generated flat $R$-modules, for some (every) $\kappa\geq\omega$.
\item\label{(Ti)} $R$ is left perfect.

\end{enumerate}
One may replace pure-universal  by "direct-summand-universal". %in (ii) and (iii). In (iv) and (v), "pure-universal" suffices.
\end{thm}
\begin{proof} (\ref{(Ti)})$\implies$(\ref{(Tii)}). Let $\kappa\geq \omega$. By (\ref{(Ti)}), every flat module is projective, hence a direct summand of a free module. But clearly every $\kappa$-generated projective module is (isomorphic to) a direct summand of $R^{(\kappa)}$, as long as $\kappa$ is infinite. 

(\ref{(Tii)})$\implies$(\ref{(Tiii)}), since pure embeddings between elementarily equivalent modules are elementary (and $T$ is complete).

(\ref{(Tiii)})$\implies$(\ref{(Ti)}). By Lemma \ref{flatmodel}, all we need is verify that every countably generated flat model of $T$ is projective. Let $M$ be such a model.  By (\ref{(Tiii)}), $M$ is purely embedded in   $R^{(\omega)}$. We claim this makes $M$ projective. It is well known that countably generated pure submodules of (even just pure-) projectives are pure-projective (see e.g.\ \cite[Prop.\,1.3.26]{P2}). But flat+pure-projective=projective, so $M$ is projective, as desired.
\end{proof}

\begin{cor}
 %Given   $\cal N =\cal M\cap {\RFlat}$, t
 The  following are equivalent.
 
\begin{enumerate}[\rm (i)]%\label{thm}
\item\label{(4iii)}  {\rm $_R(4)$\hspace{-.1em}$_{\cal M\cap {\RFlat}}$:}  $R^{(\kappa)}$ is $\prec$-universal  among the $\kappa$-generated flat models of $T$ and  $\Sigma$-pure-injective, for some (every) $\kappa\geq\omega$.  

\item\label{(4ii)}  $R^{(\kappa)}$ is pure-universal among $\kappa$-generated flat $R$-modules and  $\Sigma$-pure-injective, for some (every) $\kappa\geq\omega$. 
\item All flat $R$-modules are $\Sigma$-pure-injective.
\item\label{(4iv)}   $T$ is totally transcendental.
\item\label{(4i)}  $R$ is left $\Sigma$-pure-injective.

\end{enumerate}
\end{cor}
\begin{proof} Just add  "$\Sigma$-pure-injective" to all three conditions of the theorem and note that this property is preserved under elementary equivalence. That even (iii) holds was noted in \cite[Prop.\,15]{Rot4}, cf.\,\cite[Cor.\,14.12]{P1}.

\end{proof}

We are left with $_R(3)$. For this first note:

\begin{lem}
 $R$ is right coherent if and only if all models of\/ $T$ are flat. \end{lem} 
\begin{proof}
Let  $R$ be right coherent. By Fact \ref{elem}(1), every module elementarily equivalent to a  free module (which is always flat)  is flat as well. Conversely, if all models of $T$ are flat, so are all direct powers $_RR^I\equiv {_RR^{(I)}}$  with infinite $I$. This is equivalent to right coherence by Chase's Theorem. 
\end{proof}

The equivalence of (iii) and (v) in the next result is due to \cite[Thm.\,3.15]{KP}. %That of (iv) was added in \cite[Prop.\,2.4]{PiRo}.

\begin{cor}\label{R3}

The  following are equivalent.
 
\begin{enumerate}[\rm (i)]%\label{thm}
\item {\rm $_R(3)$:}  $T$ is totally transcendental with $\prec$-universal  model $R^{(\kappa)}$, for some (every) $\kappa\geq 2^{|L|}$.
\item $R^{(\kappa)}$ is pure-universal among $\kappa$-generated direct products of flat $R$-modules and $\Sigma$-pure-injective, for some (every)  $\kappa\geq 2^{|L|}$.%, and all flat $R$-modules are $.
\item $R^{(\kappa)}$ is saturated, for some (every) $\kappa > |L|$.
\item All members of $\cal M$ are projective/all models of $T$ are projective.
\item $R$ is left perfect and right coherent.
% (i.e.,
\end{enumerate}
\end{cor}
\begin{proof} Adding "right coherent"  to each of the conditions in the previous corollary, (v) from there becomes (v) here (cf.\,preliminaries on coherent rings). But in order to get equivalenct conditions as formulated in (i) and (ii) here, we need to make sure they do imply imply right coherence, for which if suffices to have $R^R$ flat (assume, $R$ is infinite---finite rings are coherent anyway), which is a $\kappa$-generated direct products of flat $R$-modules when $\kappa\geq |R^R|= 2^{|R|}$. This yields the equivalence of (i), (ii), and (v). 

(iv)$\implies$(v) follows by invoking the "Bass model" $F\oplus B$ from Lemma \ref{flatmodel} and the model $R^R$, as before. Fact \ref{elem}(2) implies the converse.
  
We are left with (iii). It implies (ii) (for \emph{all} $\kappa\geq 2^{|L|}$) by Proposition \ref{weaksat} (and the fact that $2^{|L|}>|L|$). For the converse, assume (ii) (for \emph{all} $\kappa\geq 2^{|L|}$)  and let $\kappa > |L|$. Then $T$ is  totally transcendental and hence  has a saturated model $M$ of power $\kappa$ by Fact \ref{tt}.  

Some care has to be taken now with the cardinals. If $\kappa\geq 2^{|L|}$, then by (ii) $M$ is elementarily embedded in $R^{(\kappa)}$. But this is pure-injective (since all models of $T$ are), and so  \cite[Prop.\,6.33]{P1} implies, $R^{(\kappa)}$ is saturated as well, as desired.

If now  $|L| < \kappa <2^{|L|}$, then  $M$ is elementarily embedded in $R^{(2^{|L|})}$, and we finish off the same way.
\end{proof}

\begin{quest} We have just proved that if $R^{(2^{|L|})}$ is a $\prec$-universal  model  of $T$, then $R$ is right coherent.
 It would be interesting to know if universality of $R^{(2^{\omega})}$ suffices. The argument given above would certainly no longer work, as there are examples of non-coherent  (necessarily uncountable) rings $R$ with $R^\omega$ flat, \cite[Beispiel 4.2]{Len}. 
 
 In such a non-coherent example, $R^{(R^R)}$ cannot be $\prec$-universal, but could there be such an example for which  $R^{(\omega)}$ nevertheless would be universal for the countably generated models?
\end{quest}

\begin{rem}
Zimmermann's example \cite[Satz 6.5]{Zim} of a commutative perfect ring that is not coherent is $\Sigma$-pure-injective, hence all flat modules are $\Sigma$-pure-injective (= totally transcendental). 
So, while $T$ has non-flat models, namely $R^R$, they are  all t.t., for $T$ is a complete theory.

Besides, this example is countable, hence $\omega$-stable, and thus has saturated, and hence universal, models in every infinite cardinality. The (proof of the) corollary shows that these can't be projective. In fact, one may reformulate the theorem in a stronger form by replacing ``The free $R$-module of some (every) infinite rank is...'' everywhere by ``there is a projective $R$-module that is...''.  
\end{rem}

This leads to

\begin{quest}
What complete theories of modules have universal models that are projective?  (And the like.)
\end{quest}

We conclude with a note on categoricity in some other classes that play a role in this work.
\begin{rem}
\begin{enumerate}[\rm(a)]
 \item \cite[Thm.\,01(2)]{M-A1}. $\RFlat$ is categorical in all higher powers iff $R$ is a matrix ring over a local ring
whose  maximal ideal is left T-nilpotent.
 \item \cite[Prop.\, 1.1]{Trl}. $\RProj$ is categorical in some higher power iff $P^{(\omega)}$ is free, for each countably generated projective module $P$. (Note, if the number $\lambda$ of isomorphism types of countably generated projectives is $ > |L|$, then "in some higher power" implies only "in all cardinals $>\lambda$".)
 \item Clearly $\cal F$ is   categorical in all higher powers.
 \end{enumerate}
\end{rem}

\section{Bass modules}
\subsection{The classical case}\label{Bass1}
A classical result of Bass says that a ring $R$ is left perfect (i.e., every flat left $R$-module is projective) precisely when $R$ has the dcc on right principal ideals, \cite[Thm.\ P]{Ba}.\footnote{T.\,Y.\,Lam says:
\emph{This switch
from left modules to right modules, albeit not new for Bass ..., is
in fact one of the inherent peculiar features of his Theorem...\,. Unfortunately,
because of this unusual switch of sides, [the] Theorem  is often misquoted in the
literature, sometimes even in authoritative sources;...\,\cite[p.\,24]{Lam}.} From the model-theoretic perspective, or likewise in the terminology of p-functors of \cite{Zim}, there is nothing unusual about this switch. In fact, there is none, if we replace \emph{right} principal ideals by  \emph{left} finitary matrix subgroups or  pp formulas  that define them---which is the thing to do as the next theorem shows. The only switch of sides then is by the (rather accidental) fact that a left pp formula defines a right ideal in any ring, which is nothing  special about perfect rings.}

We show first that the flat module Bass used for his proof is almost a model of $T$. To do so we use its 
 presentation as given in \cite{PR}. %This will facilitate the understanding of the adjustments we make in the next subsection in order to obtain the  more general result.
To this end, let $R \supseteq a_1R \supseteq a_2R \supseteq a_3R \supseteq \ldots$ be a descending chain of principal right ideals. Set $a_0=1$ and choose ring elements $b_1, b_2, b_3, \ldots$  so that $a_{n+1}= a_n b_{n+1}$  ($n\geq0$) and thus $a_n= b_1 b_2 \ldots b_n$ ($n>0$). Consider  the chain 
$A_0\to A_1\to A_2\to \ldots$ with the $A_i$, 
 copies of the module $_RR$, and connecting maps $f_i: A_i \to A_{i+1}$,  right multiplication by $b_{i+1}$. Call the direct limit, $B$, the Bass module corresponding to this chain.
%The corresponding Bass module is the limit $B$ of this direct system.

Being a direct limit of flat modules, $B$ is flat (and countably generated). Bass showed that $B$ is projective if and only  
if every descending chain of principal ideals (including the one above)  stabilizes. In other words, if there is a properly descending chain of principal right  ideals, then there is a flat module, namely $B$ above, that is not projective. %For the purpose at hand we use Bass' Theorem as stated, but emphasize that eventually we obtain a full, independent proof of it as the special case---the flat case---of our generalization of it to pure-projectives, in the next subsections.

\begin{lem}\label{flatmodel}
 $R$ is left perfect if and only if every countably generated  flat model of $T$ is projective %(if and only if the model $F\oplus B$ is projective for $F$ the free module of countably infinite rank).
 (if and only if the model $R^{(\omega)}\oplus B$ is projective). 
\end{lem}
\begin{proof} To prove the lemma, it  suffices to produce a flat model of\/ $T$ that is projective if and only if  $B$ is.
We claim, $M := F \oplus B$ is such a model, where $F= R^{(\omega)}$. First of all, $M$ is flat and it is projective iff $B$ is. So it remains to see  $M\models T$.
 
 Since all the pp indices $|\phi/\psi|$ are infinite in $F$, it suffices to verify that a pp pair $\phi/\psi$ opens up in $M$ if and only if it does in $F$. For the nontrivial implication, suppose it opens up in $M$ because it opens up in $B$. Then it has to open up in $_RR$ (a fancier way of saying this is that definable subcategories---which are defined by the closing of certain pp pairs---are closed under direct limit). But if it opens up in $R$, it does so in $F$ as well, and we are done.
\end{proof}

\begin{rem}\label{rem2}
 If the Bass module $B$ \emph{is} projective,  by Eilenberg's trick, $F \oplus B \cong F$ (for any free module $F$ of infinite rank), and so, trivially,  $F \oplus B\models T$. The point of the above argument is that it is a model---whether $B$ is projective or not.
\end{rem}

\subsection{From projective to pure-projective}\label{pureproj}
The  Bass module is a particular countably generated flat module which is projective precisely when the ring has the dcc on principal right ideals. This dcc can be reformulated as the dcc on (pp subgroups defined by) pp formulas of the form $a|x$ with $a\in R$. We are going to generalize this construction to produce a countably generated module that is pure-projective precisely when the ring, or rather its module category, satisfies the dcc on certain pp formulas. This  extends Bass' Theorem to a new realm, while returning the original result as a special case.

Recall, a module is pure-projective iff  it is a direct summand of a direct sum of finitely presented  modules, whereas projective modules are direct summands of free modules, that is, of direct sums of copies of the module $_RR$. This tells us how to \texttt{purify} classical concepts and results: by simply allowing arbitrary fp modules in place of $_RR$.

To this end, start from a class $\A$ of fp modules, $\A\subseteq\Rmod$.  (We can actually work with sets, as $\Rmod$ is skeletally small: the isomorphism types form a set.) Then set $\B := \add \A$ and $\C := \lim\B$.  In the previous section we dealt with the special case $\A = \{_RR\}$, in which $\B$ becomes the class of fg projective modules and $\C$, by a classical theorem of Lazard,  that of all flat modules. And this is the route we take to generalize Thm.\,\ref{perfect}: replace $\{_RR\}$ by such a set  $\A\subseteq\Rmod$ and let the resulting classes $\B$ and $\C$ play the same role they did in the specific example.%---this is what we mean by `replacing $_RR$ by fp modules.' 

Call a module \texttt{pure-free} if it is a direct sum of fp modules. Note the analogy: a module is (pure-) projective if and only if it is a direct summand of a (pure) free module. The role that the free module of infinite rank, $F = {_RR}^{(\omega)}$, played in the previous section, will now be taken by the pure-free module $F_\A := \bigoplus_{A\in\A} A^{(\omega)}$.

\subsection{The pure-projective case}\label{purproj} Bj\"ork strengthened Bass' Theorem by showing that the dcc on principal right ideals implies the dcc on fg right ideals. A uniform model-theoretic proof of the combined Bass--Bj\"ork result was given 
in \cite[Thm.4.2]{PR}. Its main ingredient was the description of pure-projective direct limits in terms of stabilization of certain pp formulas \cite[Lemma 3.6]{PR}. Here we extend this proof to our more general setting and thus produce a more general type of Bass module, which we use in the same fashion as in Thm.\,\ref{perfect}, \S\ref{modules}, to derive a  universality result for $F_\A = \bigoplus_{A\in\A} A^{(\omega)}$, for any given choice of $\A\subseteq\Rmod$ (and resulting $\B$, $\C$) as above, see Thm.\,\ref{pure} below.

\begin{definition}\label{Bass}
 Given a descending chain $\Gamma$   of pp formulas of fixed arity, $\phi_0\geq\phi_1\geq\phi_2\geq\ldots$, define a \texttt{Bass module $B_\Gamma$} as the direct limit of a  direct system (in fact, a chain) obtained as follows.
Choose finitely presented free realizations $(A_i, \br a_i)$ of $\phi_i$ and maps $g_i : (A_i, \br a_i) \to (A_{i+1}, \br a_{i+1})$ for all $i$, which exist because $ \br a_{i+1}$ satisfies $\phi_i$ in $A_{i+1}$. 

Consider also the corresponding maps  $f_i : A_{i}\to B_\Gamma$ and the module $F_\Gamma$ which is the direct sum of infinitely many copies of each of the $A_i$, i.e., $F_\Gamma := \bigoplus_i A_i^{(\omega)}$.
 \end{definition}
 
 Note,  $F_\Gamma$ is pure-free, hence pure-projective. Besides, it is  countably generated.
 It is not hard to see that the pp type of $f_0(\br a_0)$ is finitely generated if and only if $\Gamma$ stabilizes, i.e., if there is $i$ such that $\phi_i\leq \phi_j$ for all $j\geq i$ (uniformly so in all modules) \cite[Lemma 3.6]{PR}. This implies the second  statement of the next   proposition. The proof of the first is, mutatis mutandis, identical to the last part of the proof  of Lemma \ref{flatmodel} (just replace $_RR$ by the $A_i$).

\begin{prop}\label{Bassversion} (In the notation of Def.\ref{Bass}.)
\begin{enumerate}[\rm (1)]
\item 
 The (pure-projective) module $F_\Gamma$ is elementarily equivalent to $F_\Gamma \oplus B_\Gamma$. 
\item   If $\Gamma$ does not stabilize, $B_\Gamma$ is not Mittag-Leffler, hence not pure-projective.
\qed
 \end{enumerate}
\end{prop}

This implies an abstract version of the above universality result.

\begin{cor}
 If\/ $\Gamma$\/ does not stabilize, $F_\Gamma$ is not a universal model of its own theory (not even purely). \qed
\end{cor}

Since direct limits of flats are flat (and  flat+pure-projective = projective), we obtain Bass' result. 

\begin{cor} (In the notation of Def.\,\ref{Bass}.)
 If the $A_i$ are flat and  $\Gamma$ does not stabilize, $B_\Gamma$ is a flat module that is  not projective. \qed
\end{cor}

To see that this, as a matter of fact, implies  the more general Bj\"ork-Bass result as in  \cite[Thm.4.2]{PR} (see there for details), it remains to note that simple divisibility formulas of the form $r|x$ (with $r\in R$) or  $\frak{A}\,|\,\br x$ (with $\frak{A}$ a matching matrix) have
projective (hence even free) free (sic!) realizations. Simply  take the free module $A$ on generators $\br x$. Then $(A, \frak{A}\,\br x)$ is a free realization of $\frak{A}\,|\,\br x$ (in particular, $(R, r)$ is a free realization of $r|x$). (Note: we can't do  better  with projective $A$, as such  formulas  are the only ones that are freely realized in a projective,  \cite[Fact 2.8]{MPR}.)

\subsection{Universality of pure-free modules} We are now able to state the universality properties the pure-free module $F_\A = \bigoplus_{A\in\A} A^{(\omega)}$ has, for any given choice of $\A\subseteq \Rmod$ (and $\B=\add\A$ and $\C=\lim\B$).

Consider $\Gamma_\A$, the set of pp formulas of any finite arity that are freely realized in a module from $\add\A$. Note that this is the same as to say \emph{freely realized in a finite direct sum of modules from $\A$}. Let $T_\A$ be the $L$-theory of $F_\A$. For brevity, members of $\C\cap\Mod T_\A$ are simply called \texttt{$\C$-models} of $T_\A$.

(One can also assign, conversely, to any given set of pp formulas, $\Gamma$, a set $\A_\Gamma$  of fp modules that are free realizations of formulas from $\Gamma$. There is a Galois connection  to be extracted here.)

We have at once the following generalization of Lemma \ref{flatmodel}. Note that, without the converse at hand that Bass' Theorem provides for perfect rings, we are able to prove this only for countably generated modules. The question arises if this is true without the countability assumption.

\begin{lem}\label{Cmodel}
$\RMod$ has the dcc on formulas from $\Gamma_\A$ if and only if every countably generated $\C$-model of $T_\A$   is pure-projective (and it suffices to consider the model $F_\A\oplus B_{\Gamma_\A}$, with $B_{\Gamma_\A}$ as in   Def.\,\ref{Bass}).\qed
\end{lem} 

The telescoping theorem of \cite[Thm.\,3.7]{PR} yields the following generalization of Rem.\,\ref{rem2}. 

\begin{rem}\label{rem3}
 If the Bass module $B_{\Gamma_\A}$ \emph{is} pure-projective,   $F_\A \oplus B_{\Gamma_\A} \cong F_\A$. 
 \end{rem}

The desired generalization of Thm.\,\ref{perfect} now follows.

\begin{thm}\label{pure} 
In the above notation, the  following are equivalent.
 
\begin{enumerate}[\rm (i)]
\item  $F_\A^{(\omega)}$ is $\prec$-universal  (equivalently, pure-universal)  among the countably generated  models of  $T_\A$. %, for some (every) $\kappa\geq\omega$.
\item $F_\A^{(\omega)}$ is pure-universal among  the countably generated  modules from $\C$.%, for some (every) $\kappa\geq\omega$.
\item $\RMod$ has the dcc on formulas from $\Gamma_\A$.

\end{enumerate}
One may replace pure-universal  by "direct-summand-universal".\qed
\end{thm}

Thm.\,\ref{perfect} and Bass'  Theorem constitute the  special case where $\A = \{_RR\}$ (cf.\ proof of Lemma \ref{flatmodel}).
In the other extreme case, where $\A$ is all of $\Rmod$, we obtain
 another classical result. In that case, clearly $\B$ is the class of (all) pure-projective $R$-modules, while, by a classical theorem of Lazard, $\C=\RMod$. Clearly, $\Gamma_{\Rmod}$ is the set of all pp formulas. If all modules are Mittag-Leffler, all pp chains must stabilize. This had been derived by the second author, \cite[Cor.\,3.5]{Rot1.5}, from his model-theoretic characterization of Mittag-Leffler modules---Daniel Simson kindly informed him that he had proved this more than a decade prior to  that, \cite{Sim}. See \cite[\S4.5.1]{P2} for more on pure-semisimple rings.

\begin{cor}{\rm (D.\ Simson)}.\label{Sim}
The following are equivalent for any ring $R$.
\begin{enumerate}[\rm (i)]
 \item $R$ is left pure-semisimple, i.e., every left $R$-module is pure-projective.
 \item Every (left) $R$-module is $\Sigma$-pure-injective (i.e., totally transcendental).
 \item Every  (left) $R$-module is Mittag-Leffler. 
 \end{enumerate}
 \end{cor}
 \begin{proof}
 Only the extension to uncountably generated modules has to be verified. This, however, follows from the fact that the above dcc ensures that \emph{all} (not only the countably generated) modules are ($\Sigma$-) pure-injective, hence all are pure-projective as well.
\end{proof}
 
\begin{rem}
 If a class of modules, $\cal K$, has a pure-universal module which is projective, then every countably generated module in $\cal K$ is projective as well.
\end{rem}
\begin{proof}
 Let $M\in\cal K$ be countably generated and $P\in\cal K$ projective and pure-universal for $\cal K$. Then $M$ is purely embedded in $P$, hence flat and Mittag-Leffler (as so is $P$). Being countably generated, $M$ is in fact (flat and) pure-projective, hence projective.
\end{proof}

\noindent
Anand Pillay

Department of Mathematics University of Notre Dame Notre Dame IN 46556

USA

e-mail: Anand.Pillay.3@nd.edu\\

\noindent
Philipp Rothmaler

Department of Mathematics, The CUNY  Graduate Center, New York, NY 10016

USA

e-mail: philipp.rothmaler@bcc.cuny.edu

\end{document}